\date{September, 2017}
\documentclass[12pt]{amsart}
\usepackage{latexsym,amsmath,amsfonts,amscd,amssymb}
\usepackage{graphics}
\newtheorem{theorem}{Theorem}
\newtheorem{theorem*}{Theorem*}
\newtheorem{lemma}[theorem]{Lemma}
\newtheorem{lemma*}[theorem*]{Lemma}
\newtheorem{proposition}[theorem]{Proposition}
\newtheorem{proposition*}[theorem*]{Proposition}
\newtheorem{corollary}[theorem]{Corollary}
\newtheorem{corollary*}[theorem*]{Corollary}

\newtheorem{definition*}[theorem*]{Definition}

\theoremstyle{remark}

\newcommand{\cD}{{\mathcal D}}

\newcommand{\cO}{{\mathcal O}}

\newcommand{\CC}{{\mathbb C}}
\newcommand{\DD}{{\mathbb D}}

\newcommand{\QQ}{{\mathbb Q}}
\newcommand{\RR}{{\mathbb R}}
\newcommand{\TT}{{\mathbb T}}
\renewcommand{\SS}{{\mathbb S}}
\newcommand{\ZZ}{{\mathbb Z}}

\renewcommand{\a}{\alpha}

\newcommand{\eps}{\epsilon}

\renewcommand{\o}{\omega}

\newcommand{\D}{\Delta}

\title{On quasi-invariant curves}

\subjclass[2010]{Primary: 37 F 50, 37 F 25.}

\author[R. P\'{e}rez-Marco]{Ricardo P\'{e}rez-Marco}
\address{CNRS, IMJ-PRG, Paris 7, Bo\^\i te courrier 7012, 75005 Paris Cedex 13, France}
\email{ricardo.perez.marco@gmail.com}

\thanks{.}

\begin{document}

\begin{abstract}
Quasi-invariant curves are used in the study of hedgehog dynamics. 
Denjoy-Yoccoz lemma is the preliminary step for Yoccoz's complex renormalization techniques for the study of linearization of 
analytic circle diffeomorphisms. 
We give a geometric interpretation of Denjoy-Yoccoz lemma using the hyperbolic metric 
that gives a direct construction of quasi-invariant curves without renormalization. 
\end{abstract}

\maketitle

\section{Introduction.}

Yoccoz's approach to linearization of analytic circle diffeomorphisms (\cite{Yo3}, \cite{Yo4}) is 
based on complex sectorial renormalizations. 
These techniques were first used in his celebrated proof of the optimality 
of the Brjuno condition (\cite{PM}, \cite{Yo2}).

\medskip

The sectorial renormalization construction needs enough 
space around the circle, or, in other words, to have an analytic circle diffeomorphism that extends to an annulus of large modulus. 
For this, one needs to get first a good control on the real estimates on the Schwarzian derivative and non-linearity for 
high iterates of smooth circle diffeomorphisms that were developed by M. Herman \cite{He} and J--Ch. Yoccoz \cite{Yo1}. 
The estimates on the non-linearity, allows to control long orbits outside the circle using a Denjoy type lemma. 
This Denjoy-Yoccoz lemma is Proposition 4.4 in section 4.4 of \cite{Yo4}.

\medskip

The first application of Denjoy-Yoccoz lemma is to carry-out a sectorial renormalization in order to obtain an analytic circle diffeomorphism
which extends in a large annulus containing the circle (section 3.6 of \cite{Yo4}). The analysis of the linearization problem 
proceeds by successive renormalizations of two types. 
We have to distinguish when the rotation number is small or large compared to the inverse of the modulus of the annulus. In the first situation
with a small rotation number, the lemma is fundamental.

\medskip

\textbf{Non-linearizable dynamics.}
Sectorial renormalizations are useful in the non-linearizable situation. They were used by the author to 
study the dynamics of hedgehogs. Hedgehogs associated to an indifferent irrational non-linearizable fixed point are full non-trivial compact 
connected sets totally invariant by the dynamics. Hedgehogs were discovered by the author in \cite{PM1}. They are similar to Birkhoff topological 
invariant compact connected sets associated to Lyapunov unstable fixed points (see \cite{B}), but they are totally invariant in the holomorphic 
situation for indifferent fixed points.  
The topology of hedgehogs is 
always involved and not completely elucidated (see   \cite{B2}, \cite{B1}, \cite{PM2}). Despite this, the dynamics on the hedgehog
can be analyzed and exhibits remarkable rigidity properties. In some heuristic sense, the restriction of the dynamics to the hedgehog behaves 
as a complex automorphism of the disk with a fixed point, that is, as a rigid rotation. So, for example, the $q_n$ iterates of the dynamics, $(q_n)$ being the 
sequence of denominators of the convergents of the rotation number $\alpha$, converge uniformly
to the identity. Thus the dynamics is uniformly recurrent. 

\medskip

\begin{figure}[ht]
\centering
\resizebox{6cm}{!}{\includegraphics{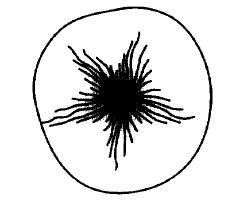}}    % name of the file - without extension
\caption{A hedgehog and its defining neighborhood.}
\end{figure}

\medskip

Hedgehogs and their dynamics are the main ingredient for the analysis of the general non-linearizable dynamics. For example, we have
the following Theorem for which no proof is known without using hedgehogs.

\begin{theorem}[\cite{PM0}, \cite{PM5}]
There is no orbit converging by positive 
or negative iteration to an indifferent irrational fixed point 
of an holomorphic map and distinct from the fixed point.
\end{theorem}

This Theorem is the key to the solution of some old problems. It solves a question of P. Fatou (1920, \cite{F}), and 
a question for singularities of differential equations in the complex domain 
due to \'E Picard (1896, \cite{P} p.30) and H. Dulac (1904, \cite{D} p. 7, case 2). The original problem was 
raised by C. Briot and J.-C. Bouquet in 1856 (\cite{BB}, section 85). 
For the relation with singularities of holomorphic foliations in $\CC^2$ in the Siegel domain we refer to \cite{PM-Yo} where we establish with 
J.-Ch Yoccoz a complete caracterization by the holonomy.

\medskip

The main tool for studying the dynamics on the hedgehog is the construction of a sequence $(\gamma_n)$ of quasi-invariant Jordan 
curves that surround and osculate the hedgehog. The Jordan domains $\Omega_n$ bounded by $\gamma_n$ are neighborhoods of the hedgehog. These 
curves are close to the hedgehog and are almost invariant by high iterates of the dynamics. Moreover, the $q_n$-iterates of the dynamics 
are close to the identity on these curves. One concludes, using the maximum principle, that the same happens on the hedgehog. Also, orbits 
near these quasi-invariant curves travel all around and form an $\epsilon_n$-dense orbit. In particular they cannot 
jump inside the domain $\Omega_n$ without $\eps_n$-visiting all points of $\gamma_n$. This is the idea behind the proof of Theorem 1.

\medskip

Quasi-invariant curves are constructed for analytic circle diffeomorphisms in a complex tubular neighborhood of the circle. The relation 
between indifferent fixed points and analytic circle diffeomorphisms was elucidated using hedgehogs by a construction 
presented in \cite{PM1}. More precisely,
for any indifferent irrational non-linearizable fixed point of an holomorphic 
map $f(z)=e^{2\pi i \a} z + \cO (z^2)$, $\a\in \RR-\QQ$, in a neighborhood of $0$ 
where $f$ and $f^{-1}$ are well defined, there exists a non-trivial (i.e. larger than the fixed point) full compact connected set $K$ which is 
totally invariant
$$
K=f(K)=f^{-1}(K) \ .
$$

The compact $K$ is a hedgehog. 

\medskip

\begin{figure}[h]
\centering
\resizebox{9cm}{!}{\includegraphics{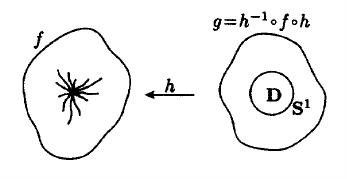}}    % name of the file - without extension
\caption{Dictionary between fixed points and circle maps.} 
\end{figure}

\medskip

We consider a conformal representation $h: \CC-\overline{\DD} \to \CC-K$ ($\DD$ is the unit disk), and we conjugate the dynamics to a univalent map 
$g$ in an annulus $V$ having the circle $\SS^1=\partial \DD$ as the inner boundary,
$$
g=h^{-1}\circ f\circ h : V \to \CC \ .
$$
One can show (\cite{PM2}, \cite{PM1}) that $K$ is not locally connected and $h$ does not extend to a continuous correspondence between $\SS^1$ and $K$. 
Nevertheless, it is easy to prove that 
$f$ extends continuously to Caratheodory's prime-end compactification of $\CC-K$. This shows that $g$ extends continuously to $\SS^1$ and its Schwarz
reflection defines an analytic map of the circle defined on $V\cup \SS^1 \cup \bar V$, where $\bar V$ is the reflected annulus of $V$. Then we can show
that $g$ is an analytic circle diffeomorphism with rotation number $\a$. 
Therefore, the dynamics in a complex neighborhood of a hedgehog corresponds to the dynamics of an analytic circle diffeomorphism.

\medskip

The properties of quasi-invariant curves for $g$ from \cite{PM0} and \cite{PM1} can be formulated using the Poincar\'e metric of 
the exterior of the closed unit disk:

\begin{theorem} (Quasi-invariant curves) \label{thm_quasi}
 Let $g$ be an analytic circle diffeomorphism with irrational rotation number $\a$. 
 Let $(p_n/q_n)_{n\geq 0}$ be the sequence of convergents of $\alpha$ given 
 by the continued fraction algorithm.
 
 Given $C_0>0$ there is  $n_0 \geq 0$ large enough such that there is a sequence of quasi-invariant curves $(\gamma_n)_{n\geq n_0}$ for $g$ which are Jordan 
 curves homotopic to $\SS^1$ and exterior to $\overline{\DD}$ such that all the iterates
 $g^j$, $0\leq j\leq q_{n}$, are defined on a neighborhood of the closure of the annulus $U_n$ bounded by $\SS^1$ and $\gamma_n$, 
 and we have 
 $$
 \cD_{P}(g^j(\gamma_n), \gamma_n) \leq C_0 \ ,
 $$
 where $\cD_P$ denotes the Hausdorff distance between compact sets associated to $d_P$, the Poincar\'e distance in $\CC-\overline{\DD}$.
 We also have for any $z\in \gamma_n$, $d_P(g^{q_{n}} (z), z) \leq C_0$, that is,
 $$
 ||g^{q_{n}} -{\hbox{\rm id}}||_{C^O_P(\gamma_n)} \leq C_0 \ .
 $$
\end{theorem}

The delicate, and useful, part of the construction of quasi-invariant curves is to obtain the estimates for the Poincar\'e metric, which is much 
stronger than the estimates for the euclidean metric since the curves $\gamma_n$ are close to $\SS^1$. 
This is also what is needed in order to transport the curves 
by the conformal representation $h$ and keep the estimates in the dynamical plane with the fixed point for $f$ for the Poincar\'e metric outside 
of the hedgehog. 
We give a new construction of quasi-invariant curves without using renormalization. It builds on the simple observation that Denjoy-Yoccoz Lemma 
has a natural hyperbolic interpretation. We carry out in this article the simpler construction 
that is sufficient for the main applications. We assume that the non-linearity of $g$ is small, that is, $|| D\log Dg ||_{C^0} <\eps_0$. The general case 
is done by carrying out a purely real renormalization as in section 3.6 of \cite{Yo4} that yields a circle map with arbitrarely small non-linearity. 
Then we transport (only once) the quasi-invariant curves by the purely real renormalization that extends to a sectorial renormalization in a small complex 
neigborhood of the circle as in \cite{PM5}.

\bigskip
\bigskip
\bigskip

\section{Analytic circle diffeomorphisms.}

\subsection{Notations.}

We denote by $\TT =\RR/\ZZ$ the abstract circle, and $\SS^1 =
E (\TT )$ its embedding in the complex plane $\CC$ given by 
the exponential mapping $E(x)=e^{2\pi i x}$.

We study analytic diffeomorphisms of the circle, but we 
prefer to work at the level of the universal covering, the real line, with 
its standard embedding $\RR \subset \CC$. We denote by $D^\omega (\TT)$ 
the space of non decreasing analytic diffeomorphisms $g$ of the real line 
such that, for any $x\in \RR$, $g(x+1)=g(x)+1$, which is the commutation to the 
generator of the deck transformations $T(x)=x+1$. An element of the space 
$D^\omega (\TT)$ has a well defined rotation number $\rho (g)\in \RR$. The order preserving 
diffeomorphism $g$ is conjugated to the rigid translation $T_{\rho(g)}: x\mapsto x+\rho(g)$, 
by an orientation preserving homeomorphism $h:\RR \to \RR$, such that $h(x+1)=h(x)+1$.

For $\Delta >0$, we note $B_{\Delta} =\{ z\in \CC ; |\Im z | < \Delta \}$, and 
$A_\Delta = E(B_\Delta )$. The subspace $D^\omega (\TT, \Delta )\subset D^\omega (\TT )$
is composed by the elements of $D^\omega (\TT)$ which extend analytically to 
a holomorphic diffeomorphism, denoted again by $g$, such that $g$ and $g^{-1}$
are defined on $B_\Delta$.

\subsection{Real estimates.}

We refer to \cite{Yo4} for the results on this section. We assume that the orientation preserving 
circle diffeomorphism $g$ is $C^3$
and that the rotation number $\alpha =\rho (g)$ is irrational. 
We consider the convergents $(p_n /q_n)_{n\geq 0}$ of $\alpha$ obtained by the continued fraction algorithm (see \cite{HW} for 
notations and basic properties of continued fractions). 

\medskip

For $n\geq 0$, we define the map $g_n (x) =g^{q_n}(x)-p_n$ and the intervals $I_n(x)=[x, g_n(x)]$,
$J_n(x)=I_n(x) \cup I_n(g_n^{-1}(x)) =[g_n^{-1}(x), g_n(x)]$. Let 
$m_n (x)=g^{q_n} (x)-x-p_n=\pm|I_n(x)|$,  $M_n =\sup_{\RR} |m_n (x)|$, and $m_n =\min_{\RR} |m_n (x)|$. 
Topological linearization is equivalent to $\lim_{n\to +\infty } M_n =0$. This is always true 
for analytic diffeomorphisms by Denjoy's Theorem, that holds for $C^1$ diffeomorphisms
such that $\log Dg$ has bounded variation. 

\medskip

Since $g$ is topologically linearizable, combinatorics of the irrational translation (or the continued fration algorithm) shows:

\begin{lemma} \label{lem_comb}
Let $x\in \RR$, $0\leq j < q_{n+1}$ and $k\in \ZZ$ the intervals $g^j\circ T^k(I_n(x))$
have disjoint interiors, and the intervals $g^j\circ T^k(J_n(x))$ cover $\RR$ at most twice.
\end{lemma}

\medskip

We have the following estimates on the Schwarzian 
derivatives of the iterates of $f$, for $0\leq j\leq q_{n+1}$,
$$
\left |S g^j (x)\right | \leq \frac{M_n e^{2V}S}{|I_n(x)|^2} \ ,
$$
with $S=||Sg||_{C^0(\RR )}$ and $V=\hbox{\rm {Var}} \log Dg$.

These estimates imply a control of the non-linearity of the iterates (Corollary 3.18 in \cite{Yo4}): 

\begin{proposition}
For $0\leq j\leq 2q_{n+1}$, $c=\sqrt{2S} e^V$,
$$
|| D \log Dg^j ||_{C^0(\RR )} \leq c \, \frac{M_n^{1/2}}{m_n} \ .
$$
\end{proposition}

These give estimates on $g_n$. More precisely we have (Corollary 3.20 in \cite{Yo4}):

\begin{proposition} \label{prop_estimate} For some constant $C >0$, we have
 $$
 ||\log Dg_n ||_{C^0(\RR )}\leq C M_n^{1/2} \ .
 $$
\end{proposition}

\begin{corollary}  \label{cor_1}
For any $\eps>0$, there exists $n_0\geq 1$ such that for $n\geq n_0$,  we have
$$
 ||Dg_n -1||_{C^0(\RR )}\leq \eps \ .
 $$
\end{corollary}

\begin{proof}
Take $n_0 \geq 1$ large enough so that for $n\geq n_0$, $C M_n^{1/2} <\min (\frac{2}{3}\eps, \frac12) $, then use Proposition \ref{prop_estimate} and 
$\left |e^w-1 \right | \leq \frac{3}{2} |w|$ for $|w|<1/2$.
\end{proof}

% In particular, we have
% \begin{proposition}  For any $\eps>0$, there exists $n_0\geq 1$ such that for $n\geq n_0$, for any $x\in \RR$ and $y\in I_n(x)$ we have
% \begin{align*}
% |m_n(x)-m_n(y)| \leq (1+\eps ) |x-y| \ .
% \end{align*}
% \end{proposition}
% 
% \begin{proof}
%  There exists $\xi \in ]x, y[$ such that 
%  $$
%  g_n(x)-g_n(y) = Dg_n (\xi) (x-y) \ .
%  $$
%  Therefore, using Proposition \ref{prop_estimate}, we get
%  \begin{align*}
%   |m_n(x)-m_n(y)|=\left | (g_n(x)-x)-(g_n(y)-y)\right | &\leq  |Dg_n(\xi )| |x-y| + |x-y| \ , \\
%   &\leq (1+C M_n^{1/2}) |x-y| \ ,
%  \end{align*}
%  and we conclude using $M_n\to 0$.
% \end{proof}

\begin{corollary}\label{cor_estimate}
For any $\eps>0$, there exists $n_0\geq 1$ such that for $n\geq n_0$, for any $x\in \RR$ and $y\in I_n(x)$ we have
$$
1-\eps \leq \frac{m_n(y)}{m_n(x)}\leq 1+\eps \ .
$$
\end{corollary}

\begin{proof}
 We have $D m_n(x)=Dg_n(x)-1$, and 
 $$
 \left | m_n(y)-m_n(x)\right | \leq ||Dm_n ||_{C^0(\RR )} |y-x| \leq ||Dg_n -1||_{C^0(\RR )} |m_n(x)| \ .
 $$
 We conclude using Lemma \ref{cor_1}.
\end{proof}

\newpage

\section {Denjoy-Yoccoz Lemma.}

Once we have these real estimates, and, more precisely, a control on the non-linearity, 
we can use them in a complex neighborhood. Using the notations introduced 
in the previous section, the raw form of the Denjoy-Yoccoz Lemma (see  \cite{Yo2}) is the following:

\begin{lemma}[Denjoy-Yoccoz Lemma]
Let $\D >0$ and $g\in D^{\o} (\TT , \D )$ holomorphic 
and continuous on $\overline {B_{\D}}$. We assume that 
$$
\tau =||D \ \log \ Dg ||_{C^0 (\overline {B_{\D}} )} <\frac1{16},
$$
and that for $n\geq 0$, 
$$
M_n \leq \frac{\D} {2 D_0},
$$
where $4< D< \frac{1}{4 \tau }$.

Let $z\in \CC$, we write $z_0=x_0+i m_n (x_0) y_0$, $y_0 \in \CC$, 
and we assume that $|y_0| \leq D_0$. 
Then for $0\leq j\leq q_{n+1}$,  we have 
$$
g^j (z_0) =f^j (x_0) +i m_n (g^j(x_0)) \ y_j,
$$
with 
$$
|y_j-y_0|\leq 3 D_0\tau  |y_0|.
$$
\end{lemma}

\begin{proof}
Let $z_j=g^j (z_0)$ and $x_j=g^j (x_0)$. We prove the Lemma by induction on $j\geq 0$.
For $j=0$ the result is obvious. Assume the result for $ i \leq j-1$, 
$|y_i|\leq 7/4  \ |y_0| \leq 2D_0 -1$, since
$3D_0\tau \leq 3/4$ and $D_0>4$. 

By the chain rule we have
$$
\log D g^j (z_0) =\sum_{l=0}^{j-1} \log D g (z_l) \ ,
$$
so
\begin{align*}
\left | \log D g^j (z_0) -\log D g^j (x_0) 
\right | &\leq \sum_{l=0}^{j-1} 
\left | \log D g (z_l) -\log D g (x_l) \right | \\
&\leq \tau \sum_{l=0}^{j-1} |z_l -x_l| \\
&\leq  \tau (2D_0-1) \sum_{l=0}^{j-1} |m_n (x_l)| \ .
\end{align*}

Considering the $j$-iterate of $g$ on the interval 
$]x_0, g^{q_n} (x_0)-p_n[$, we obtain a point 
$\zeta \in  ]x_0, g^{q_n} (x_0)-p_n[$ such that,
$$
Dg^j (\zeta ) =\frac {m_n (x_j)}{m_n (x_0)},
$$
and 
$$
\left | \log D g^j (\zeta ) -\log D g^j (x_0) 
\right | \leq \tau |m_n (x_0)|\leq \tau 
\sum_{l=0}^{j-1} |m_n (x_l)|.
$$

Adding the two previous inequalities, we have 
$$
\left | \log D g^j (z_0) -\log \frac {m_n (x_j)} 
{m_n (x_0)} \right | \leq 2 D_0 \tau 
\sum_{l=0}^{j-1} |m_n (x_l)|.
$$
The intervals $]x_l , g^{q_n} (x_l)-p_n [$, $0\leq l < q_{n+1}$  being disjoint modulo $1$, we have 
$$
\sum_{l=0}^{q_{n+1}-1} |m_n (x_l)| < 1 \ .
$$
So we obtain
$$
\left | \log D g^j (z_0) -\log \frac {m_n (x_j)}{ 
m_n (x_0) } \right | \leq 2D_0 \tau,
$$
and taking the exponential (using $|e^w-1|\leq 3/2 |w|$, 
for $|w| <1/2$, since $2D \tau <1/2$),
$$
\left | D g^j (z_0)-\frac {m_n (x_j)}{m_n (x_0) } 
\right | \leq 3 D_0\tau \frac{m_n (x_j)}{ m_n (x_0) }.
$$

This last estimate holds for any point $z_t$ in the 
rectilinear segment $[x_0, z_0]$. Integrating 
along this segment we get the definitive estimate,

$$
\left | g^j (z_0) -g^j (x_0) - i y_0  m_n (x_j) 
\right | \leq 3 D_0 \tau  |y_0|  |m_n (x_j)|.
$$
\end{proof}

\newpage

\section {Hyperbolic Denjoy-Yoccoz Lemma.}

\subsection{Flow interpolation in $\RR$.} Since $g$ is analytic, from Denjoy's Theorem we know that $g_{/\RR}$ 
is topologically linearizable, i.e. there exists an increasing 
homeomorphism $h: \RR \to \RR$, such that for $x\in \RR$,
$h(x+1)=h(x)+1$, and 
$$
h^{-1} \circ g \circ h =T_{\alpha} \ ,
$$
where $T_\alpha : \RR \to \RR$, $x\mapsto x+\alpha $.

We can embed $g$ into a topological flow on the real line $(\varphi_t)_{t\in \RR}$ defined, for $t \in \RR$,  
$\varphi_t =h\circ T_{t\alpha }\circ h^{-1}$. In general, when $g$ is not analytically
linearizable (i.e. $h$ is not analytic), the maps $\varphi_t$ are only  
homeomorphism of the real line, although for $t\in \ZZ$, $\varphi_t$ is analytic since 
for these values they are iterates of $g$.  In some cases for other values 
of $t$, $\varphi_t$ is an analytic diffeomorphism in the analytic centralizer of $g$ 
(see \cite{PM3} for more information on analytic centralizers). 
Now $(\varphi_t)_{t\in [0,1]}$ is an 
isotopy from the identity to $g$. The flow $(\varphi_t)_{t\in \RR}$ is  a one 
parameter subgroup of homeomorphisms of the real line commuting to the 
translation by $1$.

\medskip

\subsection{Flow interpolation in $\CC$.} There are different complex extensions of the flow $(\varphi_t)_{t\in \RR}$
suitable for our purposes. For $n\geq 0$, we can extend this topological 
flow to a topological flow in $\CC$ by defining, for  $z_0 =x_0 +i \, | m_n (x_0)| y_0 \in \CC$, with $x_0, y_0 
\in \RR$, 
$$
\varphi_{t}^{(n)} (z_0)=z_0(t)= \varphi_t(x_0) +i \, |m_n (\varphi_t(x_0))| y_0 \ .
$$
We denote $\Phi^{(n)}_{z_0}$ the flow line passing through $z_0$,
$$
\Phi^{(n)}_{z_0} = (\varphi_{t}^{(n)} (z_0))_{t\in \RR}.
$$

\medskip

\subsection{Hyperbolic Denjoy-Yoccoz Lemma.} We are now ready to give a geometric version of Denjoy-Yoccoz Lemma. We 
denote by $d_P$ the Poincar\'e distance in the upper half plane.

\begin{lemma}[Hyperbolic Denjoy-Yoccoz Lemma] \label{lem_DYhyp}
There exists $\eps_0 >0$ small enough universal constant such that the following holds. Let $4<D_0<\frac{1}{4\eps_0}$.
Let $\D >0$ and $g\in D^{\o} (\TT , \D )$ holomorphic 
and continuous on $\overline {B_{\D}}$ such that $||D \ \log \ Dg ||_{C^0 (\overline {B_{\D}} )} <\epsilon_0$.
Then there exists $n_0\geq 1$ such that for $n>n_0$,  we have for $z_0 \in B_{\D}$, $\Im z_0 >0$, $z_0= x_0+im_{n}(x_0) y_0$, with 
$0<y_0<D_0$, $0\leq j\leq q_{n+1}$,

$$
d_P(g^j(z_0), \varphi^{(n)}_j(z_0)) \leq C_0 \ ,
$$
for some constant $C_0 >0$.
\end{lemma}

\begin{proof}
 Since $M_n\to 0$, we choose $n_0\geq 1$ big enough so that for $n \geq n_0$ we have 
 $$
 \frac{\D}{2M_{n}} \leq D_0 < \frac{1}{4\eps_0} \ ,
 $$
 so we can use the Denjoy-Yoccoz lemma in the previous section. The Poincar\'e metric in the 
 upper half plane is given by
 $$
 |ds | =\frac{|d\xi|}{\Im \xi} \ .
 $$
 Therefore 
 \begin{align*} 
 d_P(z_j , \varphi_j^{(n)}(z_0)) \leq \int_{[z_j , \varphi_j^{(n)}(z_0)]} \frac{|d\xi|}{\Im \xi} &\leq |m_{n}(x_j)|\, |y_j-y_0| 
 \, \frac{1}{\inf_{\xi \in [z_j , \varphi_j^{(n)}(z_0)]} \Im \xi} \\
 &\leq |m_{n}(x_j)| \, |y_j-y_0| \, \frac{4}{|m_n(x_j)| \, y_0} \\
 &\leq 4\, \frac{|y_j-y_0|}{y_0} \leq 3=C_0
 \end{align*}
 where in the first line we used that $\Re y_j \geq \frac{1}{4} y_0$ which follow from $|y_j-y_0|\leq \frac{3}{4} y_0$
 that we also used in the last inequality.
\end{proof}

\section{Quasi-invariant curves.}

We prove now that the flow lines $\Phi^{(n)}_{z_0}$, with $y_0 > 1/2$ and $n\geq n_0$ for $n_0\geq 1$ large enough, are 
quasi-invariant curves. These flow lines are graphs over $\RR$. Given 
an interval $I\subset \RR$, we label $\tilde I^{(n)}$ the piece of $\Phi^{(n)}_{z_0}$ over $I$.

\medskip

\begin{figure}[ht]
\centering
\resizebox{6cm}{!}{\includegraphics{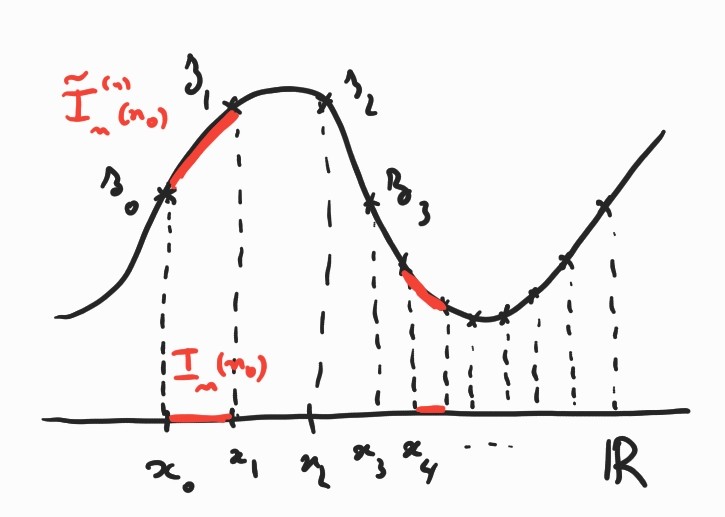}}    % name of the file - without extension
\caption{A quasi-invariant curve.}
\end{figure}

\medskip

\begin{lemma}\label{lem_bounded}
 There is $n_0\geq 1$  such that for $n\geq n_0$ and for  any $x\in \RR$, the piece $\tilde I_n^{(n)}(x)$  has bounded Poincar\'e diameter.
\end{lemma}

\begin{proof}
Let $z=x+i \, |m_n(x)| y_0$ be the current point in  $\tilde I_n^{(n)}(x)$. We have
$$
dz=\left ( 1 \pm i\, (Dg_n(x)-1) y_0 \right ) \, dx \ .
$$
For any $\eps_0 >0$,  choosing $n_0\geq 1$ large enough, for $n\geq n_0$, according to Lemma \ref{cor_1} we have
$$
\left |\frac{dz}{dx} -1 \right | \leq \eps_0 \ .
$$
Therefore, we have
$$
l_P (\tilde I_n^{(n)}(x_0))=\int_{\tilde I_n^{(n)}(x_0)} \frac{1}{|m_n(x)|\, y_0} \, |dz| \leq \int_{I_n(x_0)} \frac{1}{|m_n(x)|\, y_0} \, (1+\eps_0 )\, dx  \ .
$$
Now using Lemma \ref{cor_estimate} with $\eps= \eps_0$ and increasing $n_0$ if necessary, we have
$$
l_P (\tilde I_n^{(n)}(x))\leq \int_{I_n(x_0)} \frac{1}{|m_n(x_0)|\, y_0} \, \frac{1+\eps_0}{1-\eps_0} \, dx \leq \frac{1}{y_0} \frac{1+\eps_0}{1-\eps_0} 
\leq 2 \, \frac{1+\eps_0}{1-\eps_0}\leq C\ .
$$

\end{proof}

We assume $n\geq n_0$ from now on.

\begin{lemma}\label{lem_cover}
 For  $0\leq j < q_{n+1}$ and any $x\in \RR$, the pieces $(g^j\circ T^k(\tilde J_n^{(n)}(x)))_{0\leq j\leq q_{n+1}, k\in \ZZ}$ have bounded Poincar\'e 
 diameter and cover $\Phi^{(n)}_{z_0}$.
\end{lemma}

\begin{proof}
 From Lemma \ref{lem_bounded} any $\tilde I^{(n)}_n(x)$ has bounded Poincar\'e diameter, thus also any
 $\tilde J^{(n)}_n(x)= \tilde I^{(n)}_n(x) \cup \tilde I^{(n)}_n(g_n^{-1}(x))$. Moreover, 
 we have $g^j\circ T^k(J_n(x))=J_n(g^j\circ T^k(x))$, and all $\tilde J^{(n)}_n(g^j\circ T^k(x)) $ 
 have also bounded Poincar\'e diameter. 
 From Lemma \ref{lem_comb} these pieces cover $\Phi^{(n)}_{z_0}$.
\end{proof}

\begin{corollary}\label{cor_dense}
 For some $C_0>0$, the flow orbit $(\varphi_{j,k}^{(n)} (z_0))_{0\leq j< q_{n+1}, k\in \ZZ}$ is $C_0$-dense in $\Phi^{(n)}_{z_0}$ for the Poincar\'e metric.
\end{corollary}

We prove the first property stated in Theorem \ref{thm_quasi}:

\begin{proposition} Let $\gamma_n =\Phi^{(n-1)}_{z_0}$  for some $z_0$ from the previous lemma, then we have, for $0\leq j\leq q_n$,
$$
\cD_P(g^j(\gamma_n), \gamma_n)\leq 2C_0
$$
\end{proposition}

\begin{proof} We prove this Proposition for $n+1$ instead of $n$ (the proposition is stated to match $n$ in Theorem \ref{thm_quasi}).
It follows from the hyperbolic Denjoy-Yoccoz Lemma that the orbit $(g^j\circ T^k (z_0))_{0\leq j<q_{n+1}, k\in \ZZ}$ is $C_0$-close to flow orbit 
$(\varphi_{j,k}^{(n)} (z_0))_{0\leq j<q_{n+1}, k\in \ZZ}$, and from Corollary \ref{cor_dense} we have that a $2C_0$-neighborhood of $g^j(\gamma_{n+1})$ 
contains $\gamma_{n+1}$. Conversely, since we can chooose any $z_0 \in \gamma_{n+1}$, we also have that $g^j(\gamma_{n+1})$ is in a $C_0$-neighborhood
of $\gamma_{n+1}$.
\end{proof}

We prove the second property of Theorem \ref{thm_quasi}. We observe that $g^{q_{n+1}}(z_0) \in \tilde J_n^{(n)}(x_0)$, that $z_0 \in \tilde J_n^{(n)}(x_0)$,
and that $\tilde J_n^{(n)}(x_0)$ has a bounded Poincar\'e diameter by Lemma \ref{lem_cover}. Thus we get (taking a larger $C_0 >0$ if necessary):

\begin{proposition} For any $z_0\in  \Phi^{(n)}$ , we have 
$$
 d_P(z_0, g^{q_{n+1}}(z_0)) \leq C_0 \ .
$$
 \end{proposition}

From Lemma \ref{lem_cover} we also get the property that the hyperbolic balls 
$B_P(\varphi^{(n)}_{t+k}(z_0), C_0)$ cover $\Phi^{(n)}_{z_0}$.

\begin{proposition} We have that 
$$
U_n=\bigcup_{0\leq j< q_{n+1}, k\in \ZZ} B_P(\varphi^{(n)}_{t+k}(z_0), C_0)
$$
is a neighborhood of the flow line $\Phi^{(n)}_{z_0}$
\end{proposition}

Enlarging the constant $C_0$, we can construct a ``transient annulus'', i.e. any orbit from the outside of $\gamma_n$ that has iterates in between $\gamma_n$ and the circle 
$\SS^1$ must visit $C_0$-close any point of $\gamma_n$.

\begin{proposition} Let $(g^j(z))_{0\leq j\leq q_{n+1}}$ be an orbit that starts on a point $z$ exterior to $\gamma_n$ and has some iterate in 
between $\SS^1$ and $\gamma_n$. Then for any $w\in \gamma_n$, there is an iterate $g^j(z)$ such that
$$
d_P(g^j(z),w) \leq C_0 \ .
$$
\end{proposition}

This property is used in the proof of Theorem 1.

%\newpage


\begin{thebibliography}{9}


\bibitem{B} G.D. BIRKHOFF, \textit {Surface transformations and their dynamical 
applications}, Acta Mathematica, \textbf{43}, 1920. 


\bibitem{B2} K. BISWAS, \textit {Nonlinearizable holomorphic dynamics and hedgehogs},
Eur. Math. Soc. Newsl., \textbf{73}, p.11-15, 2009.

 	
\bibitem{B1} K. BISWAS, \textit {Positive area and inaccessible fixed points for hedgehogs},
Ergodic Theory Dynam. Systems, \textbf{36}, 6, p.1839-1850, 2016.



\bibitem{BB} C. BRIOT, T. BOUQUET, \textit{Recherches sur les propriet\'es des \'equations diff\'erentielles}, 
J. \'Ecole Imp\'eriale Polytechnique, \textbf{21 : 36}, p.133-198, 1856.


\bibitem{D} H. DULAC, \textit {Recherches sur les points singuliers des \'equations 
diff\'erentielles}, Journal de l'\'Ecole Polytechnique, II s\'erie, cahier IX, p.1-125, 1904.



\bibitem{F}P. FATOU, \textit {Sur les \'equations fonctionnelles},
Bull. Soc. Math. Fr. {\bf 47}, 
p.161-271, 1919; p.33-94, 1920; {\bf 48}, p.208-304, 1920.


\bibitem{He} M. R. HERMAN, \textit {Sur la conjugaison diff\'erentiable des 
diff\'eomorphismes du cercle \`a des rotations}, Publ. I.H.E.S. 
{\bf 49}, 1979.


\bibitem{HW} G.H.HARDY, E.M. Wright, \textit {An introduction to the theory of numbers},
4th Edition, Oxford, 1960.

\bibitem{P}\'E. PICARD, \textit {Trait\'e d'analyse}, 1st edition, tome III, Gauthier-Villars, Paris, 1896.

\bibitem{PM}R. P\'EREZ-MARCO, \textit{Solution compl\`ete au probl\`eme de Siegel de lin\'earisation 
d'une application holomorphe au voisinage d'un point fixe (d'apr\`es J.-C. Yoccoz)}, S\'em. Bourbaki, 
Exp. 753. Ast\'erisque, \textbf{206}, p.273-310, 1992. 


\bibitem{PM4}R. P\'EREZ-MARCO, \textit {Sur les dynamiques 
holomorphes non lin\'earisables et une conjecture de V. I. Arnold}, 
Ann. Scient. Ec. Norm. Sup. 4 serie, \textbf{26}, p.565-644, 1993.


\bibitem{PM2}R. P\'EREZ-MARCO, \textit {Topology of Julia sets and hedgehogs}, preprint Univ. Paris-Sud, 
94-48, 1994.
\medskip

\medskip
\bibitem{PM3}R. P\'EREZ-MARCO, \textit{Uncountable number of 
symmetries for non-linearizable holomorphic dynamics}, Inventiones 
Mathematicae, \textbf{119}, \textbf{1}, p.67-127, 1995.


\bibitem{PM0}R. P\'EREZ-MARCO, \textit{Sur une question de Dulac et Fatou}, 
Comptes Rendus Acad. Sciences, \textbf{321} , S\'erie I, p.1045-1048, 1995.
% 
% 
\bibitem{PM1}R. P\'EREZ-MARCO, \textit {Fixed points and circle maps}
Acta Mathematica,  Acta Mathematica, \textbf{179}, p.243-294 ,1997.


\bibitem{PM5} R. P\'EREZ-MARCO, \textit{Hedgehog dynamics},
manuscript, 1998.


\bibitem{PM-Yo}  R. P\'EREZ-MARCO, J.-C. YOCCOZ, \textit {Germes de feuilletages 
holomorphes \`a holonomie prescrite}, "Complex methods in dynamical systems" 
Ast\'erisque, \textbf{222}, p.345-371, 1994.


\bibitem{Yo1} J.-C. YOCCOZ, \textit {Conjugation diff\'erentiable des diff\'eomorphismes 
du cercle dont le nombre de rotation v\'erifie une condition diophantienne},
Ann. Scient. Ec. Norm. Sup., 4eme serie, \textbf{17}, p.333-359, 1984.


\bibitem{Yo2} J.-C. YOCCOZ,\textit {Th\'eor\`eme de Siegel, nombres de Bruno et polyn\^omes quadratiques} 
Ast\'erisque, \textbf{231}, p.3-88, 1995.

\bibitem {Yo3} J.-C. YOCCOZ, \textit {Lin\'earisation des diff\'eomorphismes analytiques du cercle}, 
manuscript, 1989.

\bibitem {Yo4} J.-C. YOCCOZ, \textit {Analytic linearization of circle diffeomorphisms}, 
Dynamical systems and small divisors (Cetraro, 1998),  Lecture Notes in Math., \textbf{1784}, Springer, Berlin, p.125–173, 2002.






\end{thebibliography}
\end{document}